\pdfoutput=1
\RequirePackage{ifpdf}
\ifpdf 
\documentclass[pdftex]{sigma}
\else
\documentclass{sigma}
\fi

\newtheorem{Theorem}{Theorem}
\newtheorem{Corollary}[Theorem]{Corollary}
\newtheorem{Lemma}[Theorem]{Lemma}
 { \theoremstyle{definition}
\newtheorem{Definition}[Theorem]{Definition}
\newtheorem{Remark}[Theorem]{Remark} }
\begin{document}

\allowdisplaybreaks

\newcommand{\arXivNumber}{2005.00161}

\renewcommand{\thefootnote}{}

\renewcommand{\PaperNumber}{068}

\FirstPageHeading

\ShortArticleName{Gromov Rigidity of Bi-Invariant Metrics on Lie Groups and Homogeneous Spaces}

\ArticleName{Gromov Rigidity of Bi-Invariant Metrics\\ on Lie Groups and Homogeneous Spaces\footnote{This paper is a~contribution to the Special Issue on Scalar and Ricci Curvature in honor of Misha Gromov on his 75th Birthday. The full collection is available at \href{https://www.emis.de/journals/SIGMA/Gromov.html}{https://www.emis.de/journals/SIGMA/Gromov.html}}}

\Author{Yukai SUN~$^\dag$ and Xianzhe DAI~$^\ddag$}

\AuthorNameForHeading{Y.~Sun and X.~Dai}

\Address{$^\dag$~School of Mathematical Sciences, East China Normal University,\\
\hphantom{$^\dag$}~500 Dongchuan Road, Shanghai 200241, P.R.~of China}
\EmailD{\href{mailto:52195500003@stu.ecnu.edu.cn}{52195500003@stu.ecnu.edu.cn}}

\Address{$^\ddag$~Department of Mathematics, UCSB, Santa Barbara CA 93106, USA}
\EmailD{\href{mailto:dai@math.ucsb.edu}{dai@math.ucsb.edu}}

\ArticleDates{Received May 04, 2020, in final form July 22, 2020; Published online July 25, 2020}

\Abstract{Gromov asked if the bi-invariant metrics on a compact Lie group are extremal compared to any other metrics. In this note, we prove that the bi-invariant metrics on a compact connected semi-simple Lie group $G$ are extremal (in fact rigid) in the sense of Gromov when compared to the left-invariant metrics. In fact the same result holds for a~compact connected homogeneous manifold $G/H$ with $G$ compact connect and semi-simple.}

\Keywords{extremal/rigid metrics; Lie groups; homogeneous spaces; scalar curvature}

\Classification{53C20; 53C24; 53C30}

\renewcommand{\thefootnote}{\arabic{footnote}}
\setcounter{footnote}{0}

\section{Introduction}

In \cite{2}, Gromov asks: are bi-invariant metrics on compact Lie groups extremal? (This is already problematic for ${\rm SO}(5)$.) Here a~Riemannian metric $g$ on a differentiable manifold $M$ is extremal in the sense of Gromov (not to be confused with Calabi's extremal metrics in K\"ahler geometry) if any metric~$g'$ on $M$ with $g' \geq g$ and $R_{g'} \geq R_{g}$ must have $R_{g'}= R_{g}$, where $R_g$, $R_{g'}$ denote the scalar curvature of $g$, $g'$ respectively. The metric $g$ is rigid in the sense of Gromov if in fact $g'=g$ from the conditions above.

The first result of this type is \cite{3} in which Llarull showed that the standard metric on $S^n$ is rigid. The work gives a positive answer to an earlier question of Gromov, which is motivated by Gromov--Lawson's famous work on the non-existence of positive scalar curvature metrics on the torus \cite{9}, later extended to more general class of manifolds, namely the enlargeable manifolds. In the same spirit, Llarull in fact proved that a metric on a compact manifold admitting a~$\big(1,\Lambda^2\big)$-contracting map to $S^{n}$ is rigid. Min-Oo discussed the extremality/rigidity of hermitian symmetric spaces of compact type in \cite{7}. The extremality/rigidity of complex and quaternionic projective spaces was established by Kramer~\cite{8}. Later, Goette and Semmel\-mann~\cite{6} proved that compact symmetric spaces of type $G/K$ with $\operatorname{rk}(G)-\operatorname{rk}(K)\leq 1$ are extremal (see also~\cite{16}). Then Listing improves Goette--Semmelmann's result in~\cite{4}, by weakening the extremality condition.

Note that a Lie group with a bi-invariant metric is a symmetric space, but not of the types considered above. In this short note, we present a partial positive answer to Gromov's question. Namely, we show that the bi-invariant metrics on a compact connected semi-simple Lie group~$G$ are rigid among the left-invariant metrics. More generally, we show that the normal metrics on any compact connected homogeneous space~$G/H$ without torus factor are rigid among $G$-invariant metrics on $G/H$.

\begin{Theorem}\label{main theorem}	Let $M=G/H$ be a compact homogeneous space, with $G$ a compact connected semi-simple Lie group. Then any bi-invariant metric $($also known as normal homogeneous metric$)$~$g_0$ on~$G/H$ is rigid among the $G$-invariant metrics. In other words, if $g$ is a $G$-invariant metric on $G/H$ such that $g\geq g_0$ and $R_g\geq R_{g_0}$, then $g=g_0$.
\end{Theorem}

As an immediate consequence, we have
\begin{Corollary}\label{corollary1}
	Any bi-invariant metric on a compact connected semi-simple Lie group is rigid among the left-invariant metrics.
\end{Corollary}

According to \cite{1}, if a connected Lie group admits a bi-invariant metric, it is isomorphic to the product of a compact Lie group with an abelian one. The semi-simple condition rules out the abelian factor. On the other hand, we have the famous result of Gromov--Lawson~\cite{9} and Schoen--Yau \cite{10, 11} which implies that the only metrics of nonnegative scalar curvatures on the torus are the flat ones.
\begin{Remark}
 The extremal/rigid metrics discussed here have positive scalar curvature. On the other hand, we would like to point out a related but different scalar curvature (local) extremality for K\"ahler--Einstein metrics with negative scalar curvature \cite{15}. It is an immediate consequence of Theorem~1.5 in~\cite{15} that for a K\"ahler--Einstein metric $g_0$ with negative scalar curvature on a compact complex manifold with integrable infinitesimal complex deformations, any metric $g$ sufficiently close to $g_0$ satisfying $R_g\geq R_{g_0}$ and ${\rm Vol}(g) \leq {\rm Vol}(g_0)$ must have $R_g = R_{g_0}$ (and $g$ is also K\"ahler--Einstein).
\end{Remark}

\section{Preliminaries}

Given a Riemannian manifold $(M,g)$, we denote by $R_g$ the scalar curvature of $g$.
We recall Gromov's notion of extremal/rigid metrics.
\begin{Definition}
	A metric $g_0$ on $M$ is extremal (in the sense of Gromov), if any metric $g$ on $M$ satisfying $g\geq g_0$ and $R_{g}\geq R_{g_0}$ must have identical scalar curvature, $R_g=R_{g_0}$; $g_0$ is said to be rigid (in the sense of Gromov) if the conditions above imply that $g=g_0$.
\end{Definition}

For a Lie group $G$, we denote by $\operatorname{Ad}(a)$ ($a\in G$) the adjoint action of $G$ on its Lie algebra $\mathfrak{g}$, and by $\operatorname{ad}(X)$ ($X\in \mathfrak{g}$) the induced adjoint action of $\mathfrak{g}$ on itself. In particular,
\[
\operatorname{ad}(X)Y=[X, Y], \qquad X, Y \in \mathfrak{g}.
\]
A Lie group $G$ is semi-simple if its Lie algebra $\mathfrak{g}$ is semi-simple, i.e., its Killing form
\[
K(X, Y) = {\operatorname{Tr}}(\operatorname{ad}(X) \operatorname{ad}(Y)), \qquad X, Y \in \mathfrak{g}
\]
is nondegenerate. Clearly, if $\mathfrak{g}$ is semi-simple, it has a trivial center. For a compact Lie group, the semi-simple condition is equivalent to its Lie algebra having trivial center.

If a metric on $G$ is both left-invariant and right-invariant, then it is called bi-invariant. When~$G$ is compact, bi-invariant metrics always exist. Left-invariant metrics on~$G$ are in one-to-one correspondence with inner products on its Lie algebra $\mathfrak{g}$. The following well known result plays a crucial role in the proof of our main result.

\begin{Theorem}[{\cite[Lemma~7.2]{1}}]\label{theorem2}
	In the case of a connected group $G$, a left-invariant metric is actually bi-invariant if and only if the linear transformation
	$\operatorname{ad}(X)$ is skew-adjoint with respect to the corresponding inner product, for every $X$ in the Lie algebra $\mathfrak{g}$ of~$G$.
\end{Theorem}

Now let $M=G/H$ be a compact connected homogeneous space, where $G$ is a compact connected Lie group, $H$ a closed subgroup, and the action of $G$ on $G/H$
is effective.
Let $\mathfrak{h}\subset \mathfrak{g}$ be the Lie algebra of $H$. Denote by $\operatorname{Ad}_G$ the adjoint action of $G$ on $\mathfrak{g}$ and $\operatorname{Ad}_H=\operatorname{Ad}_G|_H$ its restriction to $H$. Since $\operatorname{Ad}_H$ preserves $\mathfrak{h}$, it induces an action on $\mathfrak{g}/\mathfrak{h}$, which is equivalent to the isotropy representation of $H$.
A metric $g$ on $M=G/H$ is called $G$-invariant if it is invariant under the left action of $G$. $G$-invariant metrics on $G/H$ are naturally identified with inner products on $\mathfrak{g}/\mathfrak{h}$ which are invariant under the $\operatorname{Ad}_H$ action, see Proposition 3.16 in \cite{13}. In particular, a bi-invariant metric on $G$ gives rise to a $G$-invariant metric on $G/H$. The corresponding metric on $G/H$, usually
referred as a normal homogeneous metric on $G/H$ in the literature, will still be called bi-invariant here.

\section{Proof of the theorem}

Our proof relies crucially on a simple elegant formula for the scalar curvature for $G$-invariant metrics, as well as another lemma, in \cite{12}. We first recall this formula and the setup.

Let $g_0$ be a bi-invariant metric on $G$; still denote by $g_0$ the induced metric on $G/H$. Let $\mathfrak{g}=\mathfrak{h} + \mathfrak{m}$ be an $\operatorname{Ad}_H$ invariant decomposition orthogonal with respect to $g_0$. Then $G$-invariant metrics on $G/H$ are identified with $\operatorname{Ad}_H$-invariant inner products on $\mathfrak{m}$.

Let $\langle \cdot, \cdot \rangle_0$ be the $\operatorname{Ad}_H$-invariant inner product on $\mathfrak{m}$ corresponding to $g_0$. Let $\langle \cdot, \cdot \rangle$ be an $\operatorname{Ad}_H$-invariant inner product on $\mathfrak{m}$ inducing a $G$-invariant metric $g$ on $G/H$. Then, there is a~positive self-adjoint operator $S$ on $(\mathfrak{m},\langle X,Y\rangle_0)$ commuting with the $\operatorname{Ad}_H$-action such that
\[ \langle X,Y\rangle=\langle S(X),Y\rangle_0\]
for all $X,Y\in \mathfrak{m}$.

Since any eigenspace of $S$ is $\operatorname{Ad}_H$-invariant, there are $\operatorname{Ad}_H$-invariant subspaces $\mathfrak{m}_1, \dots , \mathfrak{m}_s$ of~$\mathfrak{m}$ such that
\begin{equation} \label{decomp}
\mathfrak{m} = \mathfrak{m}_1 \oplus \cdots \oplus \mathfrak{m}_s
\end{equation}
in orthogonal decomposition with respect to $\langle \cdot, \cdot \rangle_0$; the action of $\operatorname{Ad}_H$ on each $\mathfrak{m}_i$ is irreducible, and $S(X)=\lambda_i X$ for all $X\in\mathfrak{m}_i$, for some $\lambda_1,\dots,\lambda_s>0$. Consequently,
\[
\langle X,Y\rangle=\lambda_1 \langle X_1,Y_1\rangle_0 + \cdots + \lambda_s \langle X_s,Y_s\rangle_0,
\]
for $X=X_1+\cdots+X_s$, $Y=Y_1+\cdots+Y_s\in \mathfrak{m}$ decomposed with
respect to (\ref{decomp}). The metric $g$ is called \textit{diagonal} with respect to
the decomposition in~(\ref{decomp}).

For such metrics, there is a simple elegant formula for the scalar curvature; we refer the reader to \cite{12} for a more general discussion. Before we state this formula, let us point out the simplified situation when $M=G$.
Each $\mathfrak{m}_i$ in (\ref{decomp}) is spanned by a basis vector whenever one chooses an orthonormal basis of $\mathfrak{m}=\mathfrak{g}$ consisting of eigenvectors of~$S$. (Thus, such decompositions are by no means unique.)

 Let $\{E_{\alpha}\}$ be an orthonormal basis of $(\mathfrak{m}, \langle \,,\,\rangle_0)$ adapted to the decomposition (\ref{decomp}). We write $[E_{\alpha},E_{\beta}]_{\mathfrak{m}}=\sum_{\gamma} C_{\alpha\beta}^{\gamma}E_{\gamma}$ for some real numbers $\big\{ C_{\alpha\beta}^{\gamma} \big\}$ that we call structural constants. Here $[\ ,\ ]_{\mathfrak{m}}$ is the $\mathfrak{m}$-component of $[\ , \ ]$.
Set
\[ A_{ij}^{k}=\sum_{\alpha,\beta,\gamma}\big(C_{\alpha\beta}^{\gamma}\big)^2,\]
where the summation runs over $E_{\alpha} \in \mathfrak{m}_i$, $E_{\beta} \in \mathfrak{m}_j$, $E_{\gamma} \in \mathfrak{m}_k$.

Let $d_i=\dim \mathfrak{m}_i$. Let $B$ be the negative of the Killing form: $B(X, Y)=-K(X, Y)$. Then $B(X, X)\geq 0$, with equality if and only if $X$ is central. We define the real number $b_i$ by $B(X,Y)=b_i\langle X,Y\rangle_0$ for all $X,Y\in \mathfrak{m}_i$. Note that $b_i=0$ if and only if $\mathfrak{m}_i$ is included in the center of $\mathfrak{g}$. The following formula is equation~(1.3) in~\cite{12}.

\begin{Lemma}[{\cite[equation~(1.3)]{12}}] \label{wz1}
	Let $g$ be a $G$-invariant metric on $G/H$ with a corresponding decomposition
\eqref{decomp} as described above. Then the scalar curvature of $g$ is
	\begin{equation*} 
	R_g=\frac{1}{2} \sum_{i=1}^s \frac{b_id_i}{\lambda_i} - \frac{1}{4} \sum_{i, j, k=1}^s A^k_{ij} \frac{\lambda_{k}}{\lambda_i\lambda_{j}}.
	\end{equation*}
\end{Lemma}

The following lemma from~\cite{12} relates $b_id_i$ to the structural constants. Let
\[ C_{\mathfrak{m}_i, g_0|_{\mathfrak{h}}}= - \sum_{i=1}^h \operatorname{ad}(Z_i) \circ \operatorname{ad}(Z_i)\] be the Casimir operator of the representation of $\mathfrak{h}$ on $\mathfrak{m}_i$, where $\{Z_1, \dots, Z_h\}$ is an orthonormal basis of $(\mathfrak{h}, g_0|_{\mathfrak{h}})$ and $\operatorname{ad}(Z_i)$ should be interpreted as its restriction on $\mathfrak{m}_i$. Since $\mathfrak{m}_i$ is $\operatorname{Ad}_H$-irreducible, $C_{\mathfrak{m}_i, g_0|_{\mathfrak{h}}} =c_i \operatorname{Id}$ for some constant $c_i \geq 0$. Moreover, $c_i=0$ if and only if $\operatorname{Ad}_H$ acts trivially on $\mathfrak{m}_i$.
\begin{Lemma}[{\cite[Lemma 1.5]{16}}] \label{wz2}
	One has, for $i=1, \dots, s$,
	\begin{equation*} 
	\sum_{j,k=1}^s A^k_{ij}=b_id_i-2c_id_i.
	\end{equation*}
\end{Lemma}

\begin{Remark} \label{rflg}	Again let us look at the situation when $M=G$. In this case we choose an orthonormal basis $\{E_{i}\}_{i=1}^{n}$ of $\mathfrak{g}$ consisting of eigenvectors of $S$. Then $[E_{i},E_{j}]=\sum\limits_{k=1}^n C_{ij}^{k}E_k$ via the structure constants $C_{ij}^{k}$. The decomposition (\ref{decomp}) is given by $\mathfrak{m}_i={\rm Span}\{ E_i \}$, hence $A_{ij}^{k}=\big(C^k_{ij}\big)^2$. Moreover $c_i=0$ for all $i$. Therefore, the two lemmas above yield
\begin{equation} \label{ffsclg}
R_g= \frac{1}{4} \sum_{i, j, k=1}^n \big(C^k_{ij}\big)^2
\left[ \frac{2}{\lambda_i}- \frac{\lambda_{k}}{\lambda_i\lambda_{j}} \right].
\end{equation}
This formula can also be deduced from Koszul's formula via a direct computation.
\end{Remark}

\begin{proof}[Proof of Theorem \ref{main theorem}]
Since $\{E_{\alpha}\}$ is an orthonormal basis for $(\mathfrak{m}, \langle \cdot, \cdot \rangle_0)$, and $\langle \cdot, \cdot \rangle_0$ is bi-invariant, $C_{\alpha\beta}^{\gamma}=\langle [E_{\alpha},E_{\beta}], E_{\gamma} \rangle_0$ is skew-symmetric in all three indices by Theorem~\ref{theorem2}. Hence~$A^k_{ij}$ is symmetric in all three indices.

Now the extremal conditions $\langle X,Y\rangle\geq \langle X,Y\rangle_0$
and $R_g\geq R_{g_0}$ yield $\lambda_{i}\geq 1$ $(i=1,\dots,s)$ as well as $R_g- R_{g_0} \geq 0$. 
Lemmas~\ref{wz1} and~\ref{wz2} give
\begin{align*}
	0&\leq R_g- R_{g_0} = \frac{1}{2} \sum_{i} \frac{b_id_i}{\lambda_i} (1-\lambda_i) -\frac{1}{4} \sum_{i, j, k}A^k_{ij} \left(\frac{\lambda_{k}}{\lambda_i\lambda_{j}}-1\right) \\
	&= \sum_{i} \frac{c_id_i}{\lambda_i} (1-\lambda_i) -\frac{1}{4} \sum_{i, j, k}A^k_{ij}\left[ \frac{\lambda_{k}}{\lambda_i\lambda_{j }} +1 - \frac{2}{\lambda_i} \right].
\end{align*}

Since $c_i\geq 0$ and $d_i>0$, each term in the first summation is less than or equal to zero, with equality if and only if either $c_i=0$ or $\lambda_i=1$.

For the second summation, we use the symmetry to rewrite it as
\begin{gather*}
-\frac{1}{12} \sum_{i, j, k} A^k_{ij}\left[ \frac{\lambda_{k}}{\lambda_i\lambda_{j }} +\frac{\lambda_{i}}{\lambda_j\lambda_{k}} +\frac{\lambda_{j}}{\lambda_k\lambda_{i}} - \frac{2}{\lambda_j} - \frac{2}{\lambda_i} - \frac{2}{\lambda_k}+3 \right] \\
\qquad{} =-\frac{1}{12} \sum_{i, j, k} A^k_{ij} \frac{\lambda_i^2+\lambda_j^2+\lambda_k^2-2\lambda_i\lambda_j-2\lambda_i\lambda_k-2\lambda_k\lambda_j+3\lambda_i \lambda_j\lambda_k}{\lambda_i\lambda_j\lambda_k} .
\end{gather*}
For a fixed triple $i$, $j$, $k$, we consider the order of $\lambda_{i}$, $\lambda_{j}$, $\lambda_{k}$. Without loss of generality we can assume that $\lambda_{k}\geq \lambda_{j}\geq \lambda_{i}\geq 1$. Then the summand in the sum above can be re-organized as
\begin{gather*}
 \lambda_i^2+\lambda_j^2+\lambda_k^2-2\lambda_i\lambda_j-2\lambda_i\lambda_k-2\lambda_k\lambda_j+3\lambda_i \lambda_j\lambda_k \\
\qquad{} =  (\lambda_i-\lambda_j)^2+(\lambda_k-\lambda_j)^2+\lambda_j\lambda_k(\lambda_i-1)+\lambda_j (\lambda_k-\lambda_j)+2\lambda_i\lambda_k(\lambda_j-1)
 \geq 0
\end{gather*}
with equality if and only if $\lambda_{k}= \lambda_{j}= \lambda_{i}=1$.

But then all the inequalities become equalities. Hence, either $c_i=0$ or $\lambda_i=1$ for each $i$, and, at the same time, either $A^k_{ij}=0$ or $\lambda_{k}= \lambda_{j}= \lambda_{i}=1$ for each $(i,j,k)$.
If $\lambda_i>1$ for some $i$, then $c_i=0$, and $A^k_{ij}=0$ for all $j$, $k$. Thus $b_i=0$ by Lemma~\ref{wz2}. Therefore $\mathfrak{m}_i$ is in the center of~$\mathfrak{g}$, which contradicts the hypotheses. We conclude that $\lambda_i=1$ for all $i$, and the result follows.
\end{proof}

We end with a couple of remarks.

\begin{Remark}
From the proof, we see that if a bi-invariant metric $g_0$	 on $G/H$ is not rigid among the $G$-invariant metrics, then $G/H$ must have a torus factor. Indeed, let $\mathfrak{z} \subset \mathfrak{g}$ be the center. If for some $i$, $\lambda_i>1$, then $\mathfrak{m}_i \subset \mathfrak{z}$. Decompose $\mathfrak{g}=\mathfrak{z} + \mathfrak{g'}$ and $\mathfrak{z}=\mathfrak{m}_i + \mathfrak{k}$. Then $\mathfrak{h} \subset \mathfrak{k} +\mathfrak{g'} $. It follows that $G/H=T^{d_i} \times (K\times G')/H$.
	\end{Remark}

\begin{Remark}
It is interesting to note that the extremal conditions $g\geq g_0$ and $R_g\geq R_{g_0}$ can not be changed to the opposite inequalities. In fact, there exist $G$-invariant metrics $g$ such that $g< g_0$ and $R_g< R_{g_0}$. We illustrate the situation for $M=G={\rm SU}(2)$.

The basis $E_1=\sqrt{-1}\sigma_1$, $E_2=\sqrt{-1}\sigma_2$, $E_3=\sqrt{-1}\sigma_3$ of $\mathfrak{g}$ in terms of the Pauli spin matrices $\sigma_1$, $\sigma_2$, $\sigma_3$ satisfies
$[E_1, E_2]=2E_3$ as well as its cyclic permutations. We take $g_0$ so that $\langle X, Y \rangle_0=\frac{1}{8}B(X, Y)$, with respect to which $\{ E_1, E_2, E_3 \}$ is orthonormal.
Following the notations in Remark~\ref{rflg}, we choose $g$ so that $E_1$, $E_2$, $E_3$ are the eigenvectors with eigenvalues $\lambda_1=\lambda_2=\lambda <1$, and $\lambda_3=1/2$, respectively. Then $g<g_0$. On the other hand, by~(\ref{ffsclg}),
\[
R_g-R_{g_0}= \frac{1}{4} \sum_{i, j, k=1}^3 \big(C^k_{ij}\big)^2
\left[ \frac{2}{\lambda_i}- \frac{\lambda_{k}}{\lambda_i\lambda_{j}} -1 \right]= -\frac{1}{\lambda^2} + O\left(\frac{1}{\lambda}\right),
\]
as $\lambda \rightarrow 0^{+}$. Thus, for $\lambda$ sufficiently small, we have $R_g< R_{g_0}$.

Note that this represents the opposite rescaling of the standard sphere as compared to the example of Berger's sphere mentioned in \cite[p.~34]{14}.
\end{Remark}

\subsection*{Acknowledgements}

We are deeply grateful to Wolfgang Ziller for suggesting the more general result for the homogeneous space as well as bringing the work \cite{12} to our attention, which considerably simplifies our previous computation as well as generalizes to the more general case of homogeneous spaces. We thank Wolfgang for many helpful discussions. Thanks are also due to Professor Yurii Nikonorov for similar remarks and useful comments. Finally we thank the referee for the careful reading of the multiple versions of the paper and for many constructive suggestions which have helped improve the exposition. This research is partially supported by NSFC (Y.S.) and the Simons Foundation (X.D.).

\pdfbookmark[1]{References}{ref}
\LastPageEnding

\end{document}